\documentclass[11pt, a4paper]{amsart}

\usepackage{graphicx}
\usepackage{amsfonts}
\usepackage{amssymb}
\usepackage{amscd}
\usepackage{enumerate}
\usepackage{mathrsfs}
\usepackage{braket}

\makeatletter
\renewcommand{\section}{%
\@startsection{section}{1}%
  \z@{.7\linespacing\@plus\linespacing}{.5\linespacing}%
 {\normalfont\large\bfseries\centering}}
\makeatother

\newtheorem{theorem}{Theorem}[section]

\newtheorem{lemma}[theorem]{Lemma}
\newtheorem{proposition}[theorem]{Proposition}

\newtheorem*{theorem*}{Theorem} 
\newtheorem*{corollary*}{Corollary}
\newtheorem*{conjecture*}{Conjecture}
\newtheorem*{lemma*}{Lemma}
\newtheorem*{proposition*}{Proposition}
\newtheorem*{problem*}{Problem}
\newtheorem*{axiom*}{Axiom}
\newtheorem*{example*}{Example}
\newtheorem*{exercise*}{Exercise}

\theoremstyle{definition}

\newtheorem*{remark*}{Remark}

\newtheorem*{definition*}{Definition}

\renewcommand{\l}{\left}
\renewcommand{\r}{\right}
\newcommand{\eps}{\varepsilon}
\newcommand{\N}{{\mathbb N}}
\newcommand{\R}{{\mathbb R}}

\newcommand{\Z}{{\mathbb Z}}

\newcommand{\re}{{\rm Re}}

\newcommand{\ds}{\displaystyle}

\newcommand{\del}{\partial}

\newcommand{\wto}{\rightharpoonup}
\newcommand{\til}{\widetilde}

\newcommand{\ce}{\mathrel{\mathop:}=} 

\def\norm[#1]{\left\Vert #1 \right\Vert}
\def\tbra[#1,#2]{\left\langle #1 , #2\right\rangle} 
\def\rbra[#1,#2]{\left( #1 , #2 \right)} 
\def\sbra[#1,#2]{\left[ #1 , #2 \right]} 

\newcommand{\scA}{{\mathscr A}}
\newcommand{\scB}{{\mathscr B}}

\newcommand{\scK}{{\mathscr K}}

\newcommand{\scM}{{\mathscr M}}

\newcommand{\cE}{{\mathcal E}}

\newcommand{\cG}{{\mathcal G}}

\newcommand{\cJ}{{\mathcal J}}
\newcommand{\cK}{{\mathcal K}}

\newcommand{\cM}{{\mathcal M}}

\newcommand{\cP}{{\mathcal P}}

\newcommand{\cS}{{\mathcal S}}

\begin{document}

\title[]{Stability of algebraic solitons for nonlinear Schr\"{o}dinger equations of derivative type: variational approach}


\author{Masayuki Hayashi}
\address{Research Institute for Mathematical Sciences, Kyoto University, Kyoto 606-8502, Japan}
\curraddr{}
\email{hayashi@kurims.kyoto-u.ac.jp}
\thanks{}


\date{}

\dedicatory{}


\begin{abstract}
We consider the following nonlinear Schr\"{o}dinger equation of derivative type:
\begin{equation}
\label{eq:1}
i \partial_t u + \partial_x^2 u +i |u|^{2} \partial_x u +b|u|^4u=0 , \quad (t,x) \in \R \times \R, \ b \in \R .
\end{equation}
If $b=0$, this equation is a gauge equivalent form of well-known derivative nonlinear Schr\"{o}dinger (DNLS) equation. 
The soliton profile of the DNLS equation satisfies a certain double power elliptic equation with cubic-quintic nonlinearities. The quintic nonlinearity in \eqref{eq:1} only affects the coefficient in front of the quintic term in the elliptic equation, so the additional nonlinearity is natural as a perturbation preserving soliton profiles of the DNLS equation.
If $b>-\frac{3}{16}$, the equation \eqref{eq:1} has algebraically decaying solitons, which we call \textit{algebraic solitons}, as well as exponentially decaying solitons.
In this paper we study stability properties of solitons for \eqref{eq:1} by variational approach, and prove that if $b<0$, all solitons including algebraic solitons are stable in the energy space. 
The existence of stable algebraic solitons in \eqref{eq:1} shows an interesting mathematical example because stable algebraic solitons are not known in the context of the corresponding double power NLS.
\end{abstract}

\maketitle

\tableofcontents

\numberwithin{equation}{section} 
\section{Introduction}
\subsection{Setting of the problem}
In this paper we consider the following nonlinear Schr\"{o}dinger equation of derivative type:
\begin{equation}
 \label{eq:1.1}
i \partial_t u + \partial_x^2 u +i |u|^{2} \partial_x u +b|u|^4u=0 , \quad (t,x) \in \R \times \R, \ b \in \R .
\end{equation}
This equation has the following conserved quantities:  
\begin{align*}
\tag{Energy}
 E(u)&=\frac{1}{2}\left\| \partial_x u  \right\|_{L^2}^2 
  - \frac{1}{4}\rbra[i|u|^2\del_xu ,u]
-\frac{b}{6}\| u\|_{L^6}^6 ,
\\
\tag{Mass}
 M(u)&= \| u \|_{L^2}^2,
\\
\tag{Momentum}
P(u)&=\rbra[i\partial_x u,u],
\end{align*}
where $\rbra[\cdot,\cdot]$ is an inner product defined by
\begin{align*}
\rbra[v ,w] =\re\int_{\R} v(x)\overline{w(x)} dx\quad\text{for}~v, w\in L^2(\R ).
\end{align*}
We note that \eqref{eq:1.1} can be rewritten as 
\begin{align}
\label{eq:1.2}
i \partial_t u= E ' (u).
\end{align}
The equation \eqref{eq:1.1} is $L^2$-critical in the sense that the equation and $L^2$-norm are invariant under the scaling transformation 
\begin{align}
\label{eq:1.3}
u_{\lambda}(t,x)=\lambda^{\frac{1}{2}} u(\lambda^2 t,\lambda x), \quad \lambda >0. 
\end{align}
It is well known (see \cite{HO94a, Oz96}) that \eqref{eq:1.1} is locally well-posed in the energy space $H^1(\R)$ and that the energy, mass and momentum of the $H^1(\R)$-solution are conserved by the flow. 

When $b=0$, the equation
is a gauge equivalent form\footnote{The equation \eqref{eq:1.1} for $b=0$ and \eqref{DNLS} are equivalent under the following transformation: 
\begin{align*}
\psi(t,x) = u (t,x)\exp\l( -\frac{i}{2} \int_{-\infty}^{x} |u (t,y)|^2 dy\r).
\end{align*}
} 
of well-known derivative nonlinear Schr\"{o}dinger (DNLS) equation:
\begin{equation}
\label{DNLS}
\tag{DNLS}
i \partial_t \psi + \partial_x^2 \psi +i  \partial_x( |\psi |^{2} \psi )=0 , \quad (t,x) \in \R \times \R,
\end{equation}
which originally appeared in plasma physics as a model for the propagation of Alfv\'{e}n waves in magnetized plasma (see \cite{MOMT76, M76}). Kaup and Newell \cite{KN78} showed that \eqref{DNLS} is completely integrable, or more precisely that \eqref{DNLS} arises as a compatibility condition between two linear equations of a Lax pair. 
There is a large literature of the studies on \eqref{DNLS}, and it is beyond the scope of this paper to review it here. We just refer to \cite{H19, JLPS20} and references therein for further information.

The soliton profile of \eqref{DNLS} satisfies a double power elliptic equation with cubic-quintic nonlinearities (see \eqref{eq:1.7}). The quintic nonlinearity in \eqref{eq:1.1} only affects the coefficient in front of the quintic term in the elliptic equation, so in this sense the additional nonlinearity is not artificial, or rather natural as a perturbation preserving soliton profiles of \eqref{DNLS}. 
We note that the equation \eqref{eq:1.1} for $b\neq0$ is not expected anymore to be completely integrable while the quintic term preserves the $L^2$-critical structure of \eqref{DNLS}. Therefore, the equation \eqref{eq:1.1} can be seen as an important model to investigate the speciality of integrable structure of \eqref{DNLS} in the $L^2$-critical framework. 

We emphasize that the equation \eqref{eq:1.1} itself is an interesting mathematical model possessing a two-parameter family of solitons. For example, when $b>0$, this equation possesses both stable and unstable solitons in the $L^2$-critical framework (see \cite{O14}), which cannot be seen in other critical equations such as $L^2$-critical NLS and $L^2$-critical generalized KdV equation. The elliptic equation which soliton profiles of \eqref{eq:1.1} satisfy connects the problem on stability properties of standing waves in the double power NLS.\footnote{See \eqref{eq:1.17} below for more details.} These properties come from the rich structure of a two-parameter family of solitons.


In this paper we are interested in the equation \eqref{eq:1.1} for the case $b<0$.
The aim of this work is to study stability properties of solitons of \eqref{eq:1.1} by variational approach. We prove that if $b<0$, all solitons \textit{including algebraic solitons} are stable in $H^1(\R)$. The existence of stable algebraic solitons in \eqref{eq:1.1} shows an interesting mathematical example because stable algebraic solitons are not proved in the context of the double power NLS.



\subsection{Solitons}
\label{sec:1.2}
It is known  (see \cite{O14, H19}) that the equation \eqref{eq:1.1} has a two-parameter family of solitons. Consider solutions of (\ref{eq:1.1}) of the form
\begin{align}
\label{eq:1.4}
u_{\omega,c}(t,x)=e^{i\omega t} \phi_{\omega,c}(x-ct),
\end{align}
where $(\omega,c) \in\R^2$. 
 It is clear that $\phi_{\omega ,c}$ must satisfy the following equation:
\begin{align}
\label{eq:1.5}
-\phi ''+ \omega \phi +ic \phi ' - i|\phi|^{2} \phi '-b|\phi |^4\phi =0, \quad x \in\R.
\end{align}
Applying the gauge transformation to $\phi_{\omega ,c}$
\begin{align}
\label{eq:1.6}
\phi_{\omega,c}(x) 
&= \Phi_{\omega,c}(x) \exp\l( \frac{i}{2}cx - \frac{i}{4} \int_{-\infty}^{x} \l|\Phi_{\omega,c}(y)\r|^2 dy \r),
\end{align}
then $\Phi_{\omega ,c}$ satisfies the equation
\begin{align}
\label{eq:1.7}
- \Phi ''+ \l(\omega- \frac{c^2}{4}\r) \Phi +\frac{c}{2} |\Phi|^{2} \Phi - \frac{3}{16}\gamma |\Phi|^{4}\Phi =0, 
\quad x\in\R,
\end{align}
where $\gamma\ce 1+\frac{16}{3}b$. The positive radial (even) solution of (\ref{eq:1.7}) is explicitly obtained as follows (see also \cite{PPT79, O95}); 
if $\gamma >0$ or equivalently $b>-\frac{3}{16}$, 
\begin{align*}
\Phi_{\omega,c}^2(x) &=
\l\{
\begin{array}{ll}
\ds
 \frac{ 2(4\omega - c^2) }{\sqrt{c^2+\gamma (4\omega -c^2)} \cosh ( \sqrt{4\omega- c^2}x)-c }
& 
\ds \text{if}~-2\sqrt{\omega}<c<2\sqrt{\omega},
\\
& 
\\ 
\ds
\frac{4c}{(cx)^2+\gamma} 
& 
\ds \text{if}~c=2\sqrt{\omega},
\end{array}
\r.
\end{align*}
if $\gamma \leq 0$ or equivalently $b\leq -\frac{3}{16}$,
\begin{align*}
\Phi_{\omega,c}^2(x) &=
\begin{array}{ll}\ds
\frac{ 2(4\omega - c^2) }{\sqrt{c^2+\gamma (4\omega -c^2)} \cosh ( \sqrt{4\omega- c^2}x)-c } 
& 
\ds \text{if}~-2\sqrt{\omega} <c<-2s_{\ast}\sqrt{\omega}, 
\end{array}
\end{align*}
where $s_\ast$ is defined by
\begin{align*}
s_{\ast} =s_{\ast}(\gamma)=\sqrt{ \frac{-\gamma}{1-\gamma} } \in (0,1).
\end{align*}
Through the formula of $\Phi_{\omega ,c}$, the soliton of \eqref{eq:1.1} is explicitly represented as
\begin{align*}
u_{\omega ,c}(t,x) =\exp\l( i\omega t +\frac{i}{2}c(x-ct)-\frac{i}{4}\int_{-\infty}^{x-ct}|\Phi_{\omega ,c}(y)|^2 dy \r)\Phi_{\omega ,c}(x-ct).
\end{align*}

We note that the condition of two parameters $(\omega,c)$:
\begin{align} 
\label{eq:1.8}
\begin{array}{ll}
\ds\text{if}~\gamma >0\Leftrightarrow b>-\frac{3}{16},& \ds -2\sqrt{\omega} <c\leq 2\sqrt{\omega} ,\\[7pt]
\ds\text{if}~\gamma \leq 0\Leftrightarrow b\leq -\frac{3}{16},& \ds -2\sqrt{\omega} <c<-2s_{\ast}\sqrt{\omega} 
\end{array}
\end{align}
is a necessary and sufficient condition for the existence of non-trivial solutions of \eqref{eq:1.7} vanishing at infinity (see \cite[Theorem 5]{BeL83}). For $(\omega ,c)$ satisfying \eqref{eq:1.8}, one can rewrite
$(\omega ,c) =(\omega ,2s\sqrt{\omega})$, where the parameter $s$ satisfies 
\begin{align}
\label{eq:1.9}
\begin{array}{ll}
\ds\text{if}~b>-\frac{3}{16},& \ds -1 <s\leq 1,\\[7pt]
\ds\text{if}~b\leq -\frac{3}{16},& \ds -1 <s<-s_{\ast}. 
\end{array}
\end{align}
We note that the curve
\begin{align}
\label{eq:1.10}
\R^+ \ni\omega \mapsto (\omega , 2s\sqrt{\omega}) \in \R^2 
\end{align}
for each $s$ gives the scaling of the soliton, i.e.,
\begin{align}
\label{eq:1.11}
\phi_{\omega ,2s\sqrt{\omega}}(x) =\omega^{1/4} \phi_{1,2s} (\sqrt{\omega}x)\quad\text{for}~x\in\R,~\omega>0.
\end{align}
We note that the value $b=-\frac{3}{16}$ gives the turning point where the structure of the solitons of (\ref{eq:1.1}) changes. In particular algebraic solitons, which correspond to the case $c=2\sqrt\omega$, exist only for the case $b>-\frac{3}{16}$, which is the main interest in this paper. 

We now give the precise definition of stability of solitons in the energy space.  
\begin{definition*}
We say that the soliton $u_{\omega ,c}$ of \eqref{eq:1.1} is (orbitally) \textit{stable} in $H^1(\R)$ if for any $\eps>0$ there exists $\delta >0$ such that if $u_0\in H^1(\R)$ satisfies $\| u_0-\phi_{\omega ,c}\|_{H^1}<\delta$, then the maximal solution $u(t)$ of \eqref{eq:1.1} with $u(0)=u_0$ exists globally in time and satisfies 
\begin{align}
\label{eq:1.12}
\sup_{t\in\R}\inf_{(\theta ,y)\in\R^2}
\| u(t)-e^{i\theta}\phi_{\omega ,c}(\cdot -y)\|_{H^1}<\eps .
\end{align}
Otherwise, we say that the soliton is (orbitally) \textit{unstable} in $H^1(\R)$.\footnote{The rotations and space translations appearing in \eqref{eq:1.12} come from the invariant of the equation \eqref{eq:1.1}.}
\end{definition*}
Colin and Ohta \cite{CO06} proved the stability of exponentially decaying solitons for \eqref{DNLS}, which correspond to the solitons $u_{\omega,c}$ of \eqref{eq:1.1} for $b=0$ and $\omega>c^2/4$. 
The proof depends on variational methods related to the argument in \cite{S83} (see Section \ref{sec:1.4} for more details). Liu, Simpson and Sulem \cite{LSS13b} calculated linearized operators of a generalized derivative nonlinear Schr\"{o}dinger (gDNLS) equation:
\begin{align}
\tag{gDNLS}
\label{GD}
i\del_t u+\del_x^2 u+i|u|^{2\sigma}\del_xu=0,\quad (t,x) \in \R \times \R, ~\sigma >0,
\end{align}
and studied stability/instability of exponentially decaying solitons by applying the abstract theory of Grillakis, Shatah and Strauss \cite{GSS87, GSS90}. In particular they gave an alternative proof of the stability result in \cite{CO06} (see also \cite{GW95} for partial results in this direction). We note that the abstract theory \cite{GSS87, GSS90} is not applicable for algebraic solitons of either \eqref{DNLS}, \eqref{GD}, or \eqref{eq:1.1}, because of the lack of coercivity property of the linearized operator.    
As for algebraic solitons of \eqref{DNLS}, some kinds of stability properties were studied in \cite{KPR06, KW18} while stability/instability in the energy space remains an open problem.

The situation in \eqref{eq:1.1} for $b>0$ becomes different from the case $b=0$ due to the focusing effect from the quintic term. Ohta \cite{O14} extended the work of \cite{CO06} and proved that for each $b>0$ there exists a unique $s^*=s^*(b) \in (0,1)$ such that the soliton $u_{\omega ,c}$ is stable if $-2\sqrt{\omega}<c<2s^*\sqrt{\omega}$, and unstable if $2s^*\sqrt{\omega}<c<2\sqrt{\omega}$ (see Figure \ref{fig:1}). 
In \cite{NOW17} it was proved that algebraic soliton $u_{\omega ,2\sqrt{\omega}}$ is unstable for small $b>0$, where the assumption of smallness is used for construction of the unstable direction. If we observe momentum of solitons, the momentum is positive in the stable region, and negative in the unstable region. This implies that momentum of solitons has an essential effect on stability properties. In the borderline case $c=2s^*\sqrt{\omega}$, momentum of the soliton is zero, which corresponds to the degenerate case. Recently, in \cite{N20} instability for this case was proved for small $b>0$; later in \cite{FHpre}, instability with a large set of unstable directions was proved for all $b>0$. 

On the other hand, stability properties of solitons for the case $b<0$ seem to have been less studied. In this case momentum of all solitons is positive, which suggests that they are stable. Indeed, this is true as we show in this paper.
\begin{figure}[t]
\centering
\includegraphics[width=6cm]{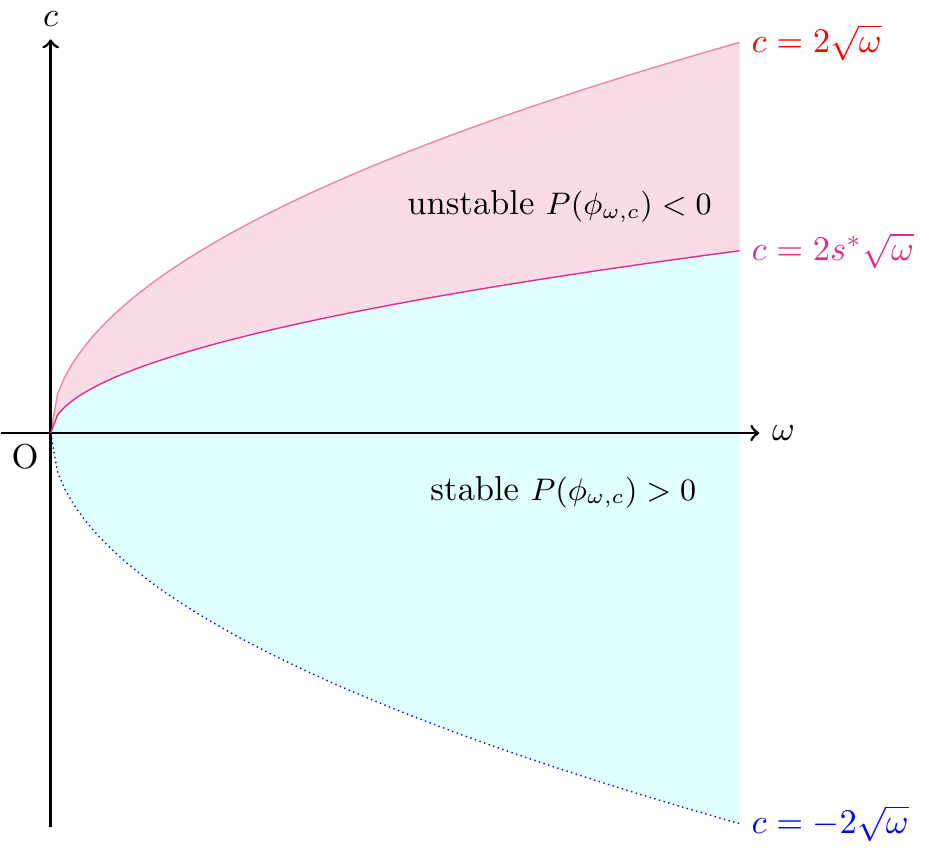}
\caption{The stable/unstable region of solitons in the case $b>0$.}
\label{fig:1}
\end{figure}

\subsection{Statement of the results}
\label{sec:1.3}
Our first theorem gives the connection between algebraic solitons and exponential decaying solitons, which would be of independent interest.
To state the result, we introduce the set $\Omega$ defined by 
\begin{align*}
\Omega = \l\{ (\omega ,c)\in\R^2 : -2\sqrt{\omega}<c<2\sqrt{\omega}\r\}.
\end{align*}
Then we have the following result.
\begin{theorem}
\label{thm:1.1}
Let $b>-\frac{3}{16}$. Suppose that $(\omega_0 ,c_0)$ satisfies $c_0=2\sqrt{\omega_0}$. Then, we have
\begin{align*}
\lim_{ \substack{(\omega ,c)\to (\omega_0 ,c_0)\\ (\omega ,c)\in\Omega} } 
\| \phi_{\omega ,c}-\phi_{\omega_0,c_0} \|_{H^m(\R )} =0
\end{align*}
for any $m\in\Z_{\geq 0}$.
\end{theorem}
\begin{remark*}
By Theorem \ref{thm:1.1} and Sobolev's embedding theorem, we obtain that 
\begin{align*}
\lim_{ \substack{(\omega ,c)\to (\omega_0 ,c_0)\\ (\omega ,c)\in\Omega} } 
\| \phi_{\omega ,c}-\phi_{\omega_0,c_0} \|_{W^{m,\infty}(\R )} =0
\end{align*}
for any $m\in\Z_{\geq 0}$.
\end{remark*}
Theorem \ref{thm:1.1} shows that algebraic solitons can be obtained as strong limits of exponentially decaying solitons. 
This relation may be useful for further study on algebraic solitons. Here we adapt the approach in \cite{H18} and give a simple proof by using explicit formulae of solitons. Recently, in \cite{FH20} a similar result of Theorem \ref{thm:1.1} was proved in the context of a double power NLS. The argument in \cite{FH20} depends on variational characterization of ground states, where explicit formulae of solitons are not necessary. 

Now we state our main result. The main result in this paper is the following stability result.
\begin{theorem}
\label{thm:1.2}
Let $-\frac{3}{16}<b<0$ and let $(\omega ,c)$ satisfy $-2\sqrt{\omega}<c\leq 2\sqrt{\omega}$. Then the soliton $u_{\omega,c}$ of \eqref{eq:1.1} is stable. In particular the algebraic soliton is stable.
\end{theorem}
\subsection{Comments on the main result}
\label{sec:1.4}
The stability result of algebraic solitons gives the counterpart of the previous instability result for the case $b>0$. 
As pointed out before, the case $c=2\sqrt{\omega}$ (regardless of $b\in\R$) cannot be treated by the abstract theory \cite{GSS87, GSS90}. 
It is difficult to study stability properties for this case, based on the study of the linearized operator $S_{\omega ,c}''(\phi_{\omega ,c})$ (see below for the definition of $S_{\omega ,c}$), because of the lack of coercivity property of $S_{\omega ,c}''(\phi_{\omega ,c})$.\footnote{The essential spectrum of $S_{\omega ,c}''(\phi_{\omega ,c})$ is given by $\sigma_{\rm ess}\l(S_{\omega ,c}''(\phi_{\omega ,c})\r)=\l[\omega-c^2/4,\infty \r)$, which gives the lack of coercivity property for the case $c=2\sqrt{\omega}$ (see \cite{FHpre} for more details).}
For the proof of Theorem \ref{thm:1.2} we use variational approach inspired from the works in \cite{S83, CO06, O14}, which enables us to treat the case $c=2\sqrt{\omega}$. 

First we review the stability theory in the papers \cite{CO06, O14}. We define the action functional $S_{\omega,c}$ by
\begin{align}
\label{eq:1.13}
S_{\omega ,c}(\phi )=E(\phi )+\frac{\omega}{2} M(\phi )+\frac{c}{2}P(\phi ),
\end{align}
and set $d(\omega ,c)=S_{\omega ,c}(\phi_{\omega ,c})$.
We note that \eqref{eq:1.5} can be rewritten as $S_{\omega ,c}' (\phi) =0$ and $\phi_{\omega ,c}$ is a critical point of $S_{\omega ,c}$.
When $b\ge 0$ the following stability result is known.
\begin{proposition}[\cite{CO06, O14}]
\label{prop:1.3}
Let $b\ge 0$ and let $(\omega ,c)$ satisfy $\omega >c^2/4$. If there exists $\xi\in\R^2$ such that
\begin{align}
\label{eq:1.14}
\tbra[d'(\omega ,c), \xi] \neq 0, ~ \tbra[d''(\omega ,c)\xi ,\xi] >0,
\end{align}
then the soliton $u_{\omega,c}$ of \eqref{eq:1.1} is stable.
\end{proposition}
Proposition \ref{prop:1.3} is proved in the following variational argument.\footnote{This can be regarded as a certain extension of the argument in \cite{S83} to a two-parameter family of solitons.}
First we prove that the soliton profile $\phi_{\omega ,c}$ is a minimizer on the Nehari manifold
\begin{align*}
\l\{ \varphi \in H^1(\R)\setminus\{ 0\} :
K_{\omega ,c}(\varphi )=0  \r\},
\end{align*}
where $K_{\omega ,c}(\varphi)\ce\l.\frac{d}{d\lambda}S_{\omega ,c}(\lambda\varphi)\r|_{\lambda=1}$.
Next we consider the potential wells
\begin{align*}
\scK_{\omega ,c}^{+}&=\l\{ u\in H^1(\R)\setminus\{0\} : S_{\omega ,c}(u) <d(\omega ,c), 
K_{\omega ,c}(u)>0\r\},\\
 \scK_{\omega ,c}^{-}&=\l\{ v\in H^1(\R)\setminus\{0\} : S_{\omega ,c}(u) <d(\omega ,c), 
K_{\omega ,c}(u)<0\r\}.
\end{align*}
By using the variational characterization on the Nehari manifold, we see that $\scK_{\omega ,c}^{+}$ and $\scK_{\omega ,c}^{-}$ are invariant under the flow of \eqref{eq:1.1}. Then, under the condition \eqref{eq:1.14}, one can control the flow around the soliton, based on the calculation of the function $\tau\mapsto d( (\omega,c)+\tau\xi)$ and properties of potential wells.

By computing $d''(\omega ,c)$ we have the following identity 
(see \cite[Lemma 1]{O14}):
\begin{align}
\label{eq:1.15}
\mathrm{det}[d''(\omega ,c)]=\frac{-2P(\phi_{\omega ,c})}{\sqrt{4\omega -c^2}\l\{c^2+\gamma (4\omega -c^2) \r\} }.
\end{align}
Here we note that $P(\phi_{\omega ,c})$ is positive if $(\omega ,c)$ satisfies that
\begin{align}
\label{eq:1.16}
\begin{array}{ll}
\ds\text{if}~b >0,& \ds -2\sqrt{\omega}<c<2s^{*}\sqrt{\omega} ,\\[7pt]
\ds\text{if}~b = 0,& \ds -2\sqrt{\omega}<c<2\sqrt{\omega}.
\end{array}
\end{align}
Therefore, we deduce that $\det d''(\omega ,c)<0$ under the condition \eqref{eq:1.16}. This yields the existence of $\xi\in\R^2$ satisfying \eqref{eq:1.14} because $d''(\omega ,c)$ has one positive eigenvalue. Hence, it follows from Proposition \ref{prop:1.3} that if \eqref{eq:1.16} holds, the soliton $u_{\omega ,c}$ is stable. This is a summary of the stability results in \cite{CO06, O14}.  

There are a few difficulties to study stability properties of solitons in the case $b<0$. 
When $b<0$ the defocusing effect from the quintic term $b|u|^4u$ gives an obstacle for the variational characterization.
To overcome that, we consider the following gauge equivalent form of \eqref{eq:1.1}:\footnote{The equation \eqref{ME} is also used in previous works \cite{Wu15, H19}.}
\begin{align}
\label{ME}
\tag{1.1$'$}
i\del_t v+\del_x^2 v+\frac{i}{2}|v|^2\del_x v-\frac{i}{2}v^2\del_x\overline{v}+\frac{3}{16}\gamma |v|^4v=0,\quad (t,x) \in \R \times \R.
\end{align}
Considering this form, one can characterize solitons on the Nehari manifold if $b\geq -\frac{3}{16}$. However, the equation \eqref{ME} does not have the ``good" Hamiltonian structure as in \eqref{eq:1.2}, so it becomes more delicate to control the flow around the soliton. Another problem arises when we treat algebraic solitons (the case $c=2\sqrt{\omega}$). We note that $d''(\omega ,c)$ does not make sense when $c=2\sqrt{\omega}$ (see \eqref{eq:1.15}) because this case corresponds to the boundary of existence region of solitons. Therefore the stability criteria in Proposition \ref{prop:1.3} does not make sense for the case $c=2\sqrt{\omega}$. 

In the present paper, we use the scaling curve \eqref{eq:1.10} effectively for the control of the flow, based on variational characterization of solitons of \eqref{ME}. This approach enables us to prove stability for algebraic solitons and exponential decaying solitons in a unified way. Also, our variational argument along the scaling curve offers new perspectives to stability theory of a two-parameter family of solitons (see the end of Section \ref{sec:4} for more details).

As a relevant work of this paper, Guo \cite{G18} studied stability of algebraic solitons of \eqref{GD} for the case $0<\sigma <1$ by variational approach.
Compared with our setting, stability problems become rather easier because the case $0<\sigma <1$ corresponds to $L^2$-subcritical problem. 
We note that the well-posedness of \eqref{GD} in $H^1(\R)$ is assumed in \cite{G18}, and the well-posedness remains an open problem in the case $0<\sigma <1$. For related topics on \eqref{GD} we refer to \cite{HO16, FHI17, LPS19} and references therein.

Algebraic solitons also appear in the following double power NLS:
\begin{align}
\label{eq:1.17}
i\del_t u+\Delta u-|u|^{p-1}u+|u|^{q-1}u=0,\quad (t,x)\in\R\times\R^d,
\end{align}
where $1<p<q<1+4/(d-2)_+$. If we consider the standing wave (soliton) solution
$
e^{i\omega t}\phi_{\omega}(x),
$
then $\phi_{\omega}$ satisfies the elliptic equation
\begin{align}
\label{eq:1.18}
-\Delta \phi +\omega\phi +|\phi|^{p-1}\phi -|\phi|^{q-1}\phi =0,\quad x\in\R^d.
\end{align}
We note that the equation \eqref{eq:1.7} for $0<c\leq 2\sqrt{\omega}$ and $\gamma>0$ corresponds to \eqref{eq:1.18} for $p=3$, $q=5$ and $d=1$.\footnote{Although there is a link of soliton profiles in between \eqref{eq:1.1} and \eqref{eq:1.17}, stability properties may change (see \cite[Remark 1]{CO06}).} Due to the defocusing effect from the lower power order nonlinearity, \eqref{eq:1.18} has algebraically decaying standing waves with $\omega=0$ as well as usual standing waves decaying exponentially with $\omega>0$. Instability and strong instability of these two types of standing waves were studied in \cite{FH20} (see also \cite{IK93, O95} for earlier results), where variational characterization of ground states plays a key role in the proof.
We remark that in the context of \eqref{eq:1.17}, stable algebraic solitons are not known for any cases of $(p,q)$.

Stability of solitions are closely related to the mass condition yielding global solutions of \eqref{eq:1.1} in the energy space. We define the mass threshold value as
\begin{align}
\label{eq:1.19}
M^*(b)=
\l\{
\begin{array}{ll}
M(\phi_{1,2s^*(b)}) &\text{if}~b> 0,\\[3pt]
\frac{4\pi}{\gamma^{\frac{3}{2}}}  &\text{if}~-\frac{3}{16}<b\leq 0(\Leftrightarrow 0<\gamma\le1).
\end{array}
\r.
\end{align}
In \cite{H19} the author obtained the new mass condition for \eqref{eq:1.1} such that if the initial data $u_0\in H^1(\R)$ of \eqref{eq:1.1} satisfies $M (u_0) < M^*(b)$, then the corresponding $H^1(\R)$-solution is global and bounded.\footnote{If $b\le -\frac{3}{16}$, for any initial data $u_0\in H^1(\R)$, the corresponding $H^1(\R)$-solution is global and bounded.}
For the case $b=0$, this mass condition is nothing but $4\pi$-mass condition, which was first proved in \cite{Wu15}.
We note that when $b>0$, $M^*(b)$ is the mass of the soliton lying in the borderline case in the stable/unstable region.
On the other hand, when $-\frac{3}{16}<b<0(\Leftrightarrow 0<\gamma<1)$, we have the following relation:
\begin{align*}
M (\phi_{1,2})  = \frac{4\pi}{\sqrt{\gamma}}<
\frac{4\pi}{\gamma^{\frac{3}{2}}} =M(\phi_{1,2})+P(\phi_{1,2}),
\end{align*}
which indicates that positive momentum of algebraic solitons boosts the threshold value. This fact and the global result above are compatible with the stability of algebraic solitons because the stability implies that the flow around algebraic solitons is global and bounded.

In the recent progress of studies on \eqref{DNLS}, global well-posedness without the smallness assumption of the mass was established in weighted Sobolev spaces by taking advantage of completely integrable structure (see \cite{PSS17, JLPS20}).\footnote{After this work was completed, it was proved in \cite{BP} that \eqref{DNLS} is globally well-posed in Sobolev spaces without the smallness assumption of the mass.} 
These results give a remarkable difference with other critical equations such as $L^2$-critical NLS and $L^2$-critical generalized KdV, while the dynamics of \eqref{DNLS} in the energy space is not yet clear including the fundamental problem of stability/instability of algebraic solitons. 
%
\subsection{Stability results for the case $b\le -\frac{3}{16}$}
The proof of Theorem \ref{thm:1.2} is not applicable to the case $b< -\frac{3}{16}$ because the argument depends on variational characterization on the Nehari manifold, which does not hold for this case.
However, by using another variational approach inspired from Cazenave and Lions \cite{CL82}, we obtain the following result.
\begin{theorem}
\label{thm:1.4}
Let $b\leq -\frac{3}{16}$ and let $(\omega ,c)$ satisfy $-2\sqrt{\omega}<c< -2s_*\sqrt{\omega}$. Then the soliton $u_{\omega,c}$ of \eqref{eq:1.1} is stable.
\end{theorem}
It may be somewhat new to apply the approach of \cite{CL82} to a two-parameter family of solitons. The key point in the proof is to solve a certain variational problem with mass constraint. To this end we consider the gauge equivalent form \eqref{ME} again. If velocity of the soliton of \eqref{ME} is negative, one can prove that the soliton is a solution of the minimization problem with mass constraint. Since velocity of all solitons for the case $b\leq -\frac{3}{16}$ is negative, we can apply this variational argument to prove stability of these solitons.
We note that the proof of Theorem \ref{thm:1.4} still works for the case $b>-\frac{3}{16}$ and $-2\sqrt{\omega} <c<0$.

One can also apply the abstract theory of \cite{GSS87, GSS90} to exponentially decaying solitons, based on spectral analysis of linearized operators. However, as can be seen in \cite{LSS13b, FHpre}, the calculation of  linearized operators for \eqref{eq:1.1} is complex because the nonlinearity contains derivative. We note that our variational proofs of Theorem \ref{thm:1.2} and \ref{thm:1.4} do not need any calculation of linearized operators.

\subsection{Organization of the paper}
The rest of this paper is organized as follows.
In Section \ref{sec:2}, we recall the fundamental properties of a two-parameter family of solitons of \eqref{eq:1.1} which are used throughout the paper. In Section~\ref{sec:3}, we study the connection between algebraic solitons and exponentially decaying solitons, and prove Theorem \ref{thm:1.1}. In Section~\ref{sec:4}, we study stability of two types of solitons for the case $-\frac{3}{16}<b<0$, and prove Theorem \ref{thm:1.2}. The key claim in the proof is Proposition \ref{prop:4.4}, where we control the flow around the solitons by using the scaling curve \eqref{eq:1.10} effectively. Finally, in Section~\ref{sec:5} we study stability of solitons with negative velocity, and prove Theorem \ref{thm:1.4}.

\section{Preliminaries}
\label{sec:2}
In this section we organize the fundamental properties of solitons of \eqref{eq:1.1}. We refer to \cite{H19} for the proof of the results in this section. 

In next sections, we mainly use the equation \eqref{ME} which is a gauge equivalent form of \eqref{eq:1.1}. Therefore we state the properties of solitons of \eqref{ME}, which also yield the properties of solitons of \eqref{eq:1.1} through the gauge transformation. We first note that \eqref{ME} is transformed from \eqref{eq:1.1} through the gauge transformation 
\begin{align*}
v(t,x)=\cG(u)(t,x)\ce u (t,x)\exp\l( \frac{i}{4} \int_{-\infty}^{x} |u (t,y)|^2 dy\r).
\end{align*}
The equation \eqref{ME} has the following conserved quantities: 
\begin{align*}
\tag{Energy}
 \cE (v) &=\frac{1}{2} \| \del_x v\|_{L^2}^2-\frac{\gamma}{32} \| v\|_{L^6}^6 ,
\\
\tag{Mass}
 \cM (v)&=\| v \|_{L^2}^2,
\\
\tag{Momentum}
\cP (v)&=\rbra[i\del_x v,v] +\frac{1}{4} \| v\|_{L^4}^4.
\end{align*}
We note that the well-posedness in $H^1(\R)$ for each of \eqref{eq:1.1} and \eqref{ME} is equivalent because $u\mapsto\cG(u)$ is locally Lipschitz continuous on $H^1(\R)$.

Let $(\omega ,c )$ satisfy \eqref{eq:1.8}. A two-parameter family of solitons of \eqref{ME} is given by
\begin{align}
\label{eq:2.1}
v_{\omega ,c}(t ,x) =\cG(u_{\omega ,c})(t,x)=e^{i\omega t}\varphi_{\omega ,c}(x-ct),
\end{align}
where $\varphi_{\omega ,c}$ is represented as
\begin{align}
\label{eq:2.2}
\varphi_{\omega ,c} (x) =e^{\frac{i}{2}cx}\Phi_{\omega ,c}(x),
\end{align}
where $\Phi_{\omega ,c}(x)$ is defined by the formulae below \eqref{eq:1.7}.
We note that $\varphi_{\omega ,c}$ satisfies the equation
\begin{align}
\label{eq:2.3}
 -\varphi'' +\omega\varphi +ic\varphi' +\frac{c}{2}|\varphi |^2\varphi
-\frac{3}{16}\gamma |\varphi|^4 \varphi =0,\quad x\in\R.
\end{align}
We define the action functional with respect to \eqref{ME} by
\begin{align*}
\cS_{\omega ,c} (\varphi ) =\cE (\varphi )+\frac{\omega}{2}\cM (\varphi )
+\frac{c}{2}\cP (\varphi ).
\end{align*}
We note that \eqref{eq:2.3} can be rewritten as $\cS_{\omega ,c} ' (\varphi )=0$ and $\varphi_{\omega ,c}$ is a critical point of $\cS_{\omega ,c}$. Concerning the conserved quantities we have the relation
\begin{align*}
\cE (\cG (u) ) =E(u),
~ \cM (\cG (u)) =M(u),~ 
\cP (\cG (u)) =P(u),
\end{align*}
which yields that 
\begin{align}
\label{eq:2.4}
\cS_{\omega ,c} (\varphi_{\omega ,c}) 
=\cS_{\omega ,c} (\cG(\phi_{\omega ,c})) 
=S_{\omega ,c} (\phi_{\omega ,c}) =d(\omega ,c).
\end{align}
In the same way as \eqref{eq:1.11}, for the parameter $s$ satisfying \eqref{eq:1.9} we have 
\begin{align*}
\varphi_{\omega ,2s\sqrt{\omega}}(x) =\omega^{1/4} \varphi_{1,2s} (\sqrt{\omega}x)\quad\text{for}~x\in\R,
\end{align*}
which implies that
\begin{align*}
\cE(\varphi_{\omega ,2s\sqrt{\omega}}) =\omega \cE(\varphi_{1,2s}),~
\cM(\varphi_{\omega ,2s\sqrt{\omega}})=\cM(\varphi_{1,2s}),~
\cP(\varphi_{\omega ,2s\sqrt{\omega}}) =\sqrt{\omega}\cP(\varphi_{1,2s}).
\end{align*}
In particular we have
\begin{align}
\label{eq:2.5}
d(\omega ,2s\sqrt{\omega}) =\omega d(1,2s).
\end{align}

Concerning mass of the solitons we have the following result.  
\begin{lemma}
\label{lem:2.1}
Let $(\omega ,c)$ satisfy \eqref{eq:1.8}. Then we have
\begin{align*}
\cM \l( \varphi_{\omega ,c}\r)  =
\l\{ 
\begin{aligned}
&~\frac{8}{\sqrt{\gamma}} \tan^{-1} \sqrt{ \frac{1+\alpha}{1-\alpha} } &&\text{if}~\gamma>0,\\
&~\frac{4\sqrt{4\omega -c^2}}{-c} &&\text{if}~\gamma=0,\\
&~\frac{4}{\sqrt{-\gamma}} \log \l( -\alpha +\sqrt{\alpha^2 -1}\r) &&\text{if}~\gamma<0,
\end{aligned}
\r.
\end{align*}
where $\alpha \ce c\l( c^2+\gamma (4\omega -c^2)\r)^{-1/2}$.
Furthermore, each of the functions
\begin{align*}
(-1, 1] \ni s \mapsto \cM\l(\varphi_{1,2s} \r) \in \l( 0,  \frac{4\pi}{ \sqrt{\gamma} }\r]\quad \text{if}~\gamma>0
\end{align*}
and 
\begin{align*}
\l( -1, -s_{\ast} \r) \ni s \mapsto \cM\l(\varphi_{1,2s}\r) \in ( 0,  \infty )
\quad \text{if}~\gamma\le 0
\end{align*}
is continuous, strictly increasing and surjective. 
\end{lemma}
By Lemma \ref{lem:2.1} and elementary calculations we have the following claim which is useful to study stability of the soliton with $c<0$.
\begin{lemma}
\label{lem:2.2}
Let $(\omega ,c)$ satisfy \eqref{eq:1.8} and $\omega >c^2/4$. Then we have
\begin{align*}
\del_{\omega}\cM( \varphi_{\omega,c}) =
\frac{-8c}{\sqrt{4\omega -c^2}\l\{c^2+\gamma (4\omega -c^2) \r\} }.
\end{align*}
\end{lemma}
The momentum of the solitons is represented as follows.
\begin{lemma}
\label{lem:2.3}
Let $(\omega ,c)$ satisfy \eqref{eq:1.8}. Then we have
\begin{align}
\label{eq:2.6}
\cP(\varphi_{\omega ,c})=
\l\{
\begin{aligned}
&\frac{c}{2} \l( -1+ \frac{1}{\gamma} \r) \cM(\varphi_{\omega ,c}) +\frac{2}{\gamma}\sqrt{4\omega -c^2} &&\text{if}~\gamma\gtrless 0,
\\
&-\frac{2\omega +c^2}{3c} \cM(\varphi_{\omega ,c}) &&
\text{if}~\gamma=0.
\end{aligned}
\r.
\end{align}
\end{lemma}
Positivity of momentum of the solitons plays an essential role in the stability theory. Concerning the sign of the momentum we have the following result.
\begin{proposition}
\label{prop:2.4}
Let $s$ satisfy \eqref{eq:1.9}. Then the following properties hold:
\begin{enumerate}[\rm (i)]
\item  If $b>0$, there exists a unique $\thickmuskip=1mu s^{*}=s^*(b)\in (0,1)$ such that $\thickmuskip=0mu\thinmuskip=0mu  \cP(\varphi_{1,2s^*})=0$. Moreover, we have $\cP(\varphi_{1,2s})>0$ for $s\in (-1, s^*)$ and $\cP(\varphi_{1,2s})<0$ for $s\in(s^{*},1]$. 

\item If $b =0$, $\cP(\varphi_{1,2s}) >0$ for $s\in (-1,1)$ and $\cP(\varphi_{1,2})=0$.

\item If $b<0$, $\cP(\varphi_{1,2s})>0$ for any $s$.

\end{enumerate}
\end{proposition}
Finally we state the energy of the solitons. The following claim is an immediate consequence of the Pohozaev identity. 
\begin{lemma}
\label{lem:2.5}
Let $s$ satisfy \eqref{eq:1.9}. Then we have
\begin{align*}
\cE(\varphi_{1,2s}) =-\frac{s}{2} \cP(\varphi_{1, 2s}).
\end{align*}
\end{lemma}

\section{Connection between two types of the solitons}
\label{sec:3}
In this section we study connection between algebraic solitons and exponential decaying solitons, and prove Theorem \ref{thm:1.1}. From the scaling relation \eqref{eq:1.11}, it is enough to discuss the convergence of $\phi_{1,2s}$ as $s\to 1$. First we prove the pointwise convergence.
\begin{proposition}
\label{prop:3.1}
 Let $b>-\frac{3}{16}$. For any $x\in\R$ we have 
\begin{align*}
\lim_{s\to 1-0}\phi_{1,2s}(x) =\phi_{1,2} (x).
\end{align*}
\end{proposition}
\begin{proof}
Fix any $x\in\R$. 
It is enough to prove that
\begin{align}
\label{eq:3.1}
\lim_{s\to 1-0}\Phi_{1,2s}(x) =\Phi_{1,2} (x),
\end{align}
because of the relation \eqref{eq:1.6} and the dominated convergence theorem.
From the explicit formula of $\Phi_{\omega,c}$, we have
\begin{align}
\label{eq:3.2}
\Phi_{1,2s}^2(x) = \frac{4(1-s^2)}{ \sqrt{s^2+\gamma (1-s^2)}\cosh\l( 2\sqrt{1-s^2} x\r) -s  }
\end{align}
for $s \in (-1, 1)$. 
By the Taylor expansion of $x\mapsto\cosh x$ around zero, the denominator is rewritten as
\begin{align}
\label{eq:3.3}
\sqrt{s^2+\gamma (1-s^2)}\l( 1+2(1-s^2)x^2+O\l( (1-s^2)^2\r)\r) -s.
\end{align}
By the Taylor expansion of the function $h\mapsto\sqrt{s^2+h}$ around zero, we have
\begin{align*}
\sqrt{s^2+\gamma (1-s^2)} -s =\frac{\gamma}{2s}(1-s^2)+O\l( (1-s^2)^2\r),
\end{align*}
which is valid for $s\in(0,1)$. Thus we have
\begin{align*}
(\ref{eq:3.3})&=\frac{\gamma}{2s}(1-s^2)+2(1-s^2)\sqrt{s^2+\gamma (1-s^2)}x^2 +O\l( (1-s^2)^2\r)\\
&=(1-s^2) \l( \frac{\gamma}{2s}+2\sqrt{s^2+\gamma (1-s^2)}x^2 +O\l( 1-s^2\r) \r).
\end{align*}
We note that the numerator and denominator share a common factor $1-s^2$. Therefore we deduce that
\begin{align*}
\Phi_{1,2s}^2(x) &=\frac{4}{ \frac{\gamma}{2s}+2\sqrt{s^2+\gamma (1-s^2)}x^2 +O\l( 1-s^2\r) } \\
&\hspace{-10pt}\underset{s\to 1-0}{\longrightarrow} \frac{8}{\gamma +4x^2} =\Phi_{1,2}^2(x),
\end{align*}
which proves \eqref{eq:3.1}.
\end{proof}
To complete the proof of Theorem~\ref{thm:1.1}, we effectively use the Br\'ezis--Lieb lemma:
\begin{lemma}[{\cite{BL83}}]
\label{lem:3.2}
Let $1\leq p < \infty$. Let $\{f_n\}_{n\in\N}$ be a bounded sequence in $L^p(\R)$ and $f_n \to f$ a.e. in $\R$ as $n\to \infty$. Then we have 
\begin{align*}
\| f_n\|_{L^p}^p - \| f_n-f\|_{L^p}^p - \| f \|_{L^p}^p \to 0
\end{align*}
as $n \to \infty$. 
\end{lemma}
\begin{proof}[Proof of Theorem \ref{thm:1.1}]
From Lemma \ref{lem:2.1} and Proposition \ref{prop:3.1}, we have
\begin{align*}
&\lim_{s\to 1-0}\phi_{1,2s}(x) =\phi_{1,2} (x) ~\text{for all}~x\in\R, \\
&\lim_{s\to 1-0} \l\| \phi_{1,2s}\r\|_{L^2}^2  =\l\| \phi_{1,2}\r\|_{L^2}^2 .
\end{align*}
Applying Lemma \ref{lem:3.2}, we have
\begin{align*}
\lim_{s\to 1-0} \| \phi_{1,2s} -\phi_{1,2}\|_{L^2}^2 =0.
\end{align*}
In the same way we also have
\begin{align}
\label{eq:3.4}
\lim_{s\to 1-0} \| \Phi_{1,2s} -\Phi_{1,2}\|_{L^2}^2 =0.
\end{align}
Here we recall that $\Phi_{1,2s}$ is a solution of the equation
\begin{align}
\label{eq:3.5}
-\Phi'' +(1-s^2) \Phi +s|\Phi|^2\Phi -\frac{3}{16}\gamma |\Phi|^4\Phi  =0,\quad x\in\R.
\end{align}
We note that
\begin{align*}
\| \Phi_{1,2s}\|_{L^{\infty}}^2 &=\Phi^2_{1,2s} (0) \\
 &= \frac{ 4(1-s^2) }{ \sqrt{s^2+\gamma (1-s^2)}-s } \\
 &=\frac{4}{\gamma} \l( \sqrt{s^2+\gamma (1-s^2)}+s \r).
\end{align*}
This formula yields that the function $(-1,1)\ni s\mapsto \| \Phi_{1,2s}\|_{L^{\infty}}$ is strictly increasing and
\begin{align*}
\lim_{s\to 1-0} \| \Phi_{1,2s}\|_{L^{\infty}}^2 =\frac{8}{\gamma} = \| \Phi_{1,2}\|_{L^{\infty}}^2.
\end{align*}
In particular we have
\begin{align}
\label{eq:3.6}
\max_{s\in (-1,1]} \| \Phi_{1,2s}\|_{L^{\infty}} = \| \Phi_{1,2}\|_{L^{\infty}}.
\end{align}
By \eqref{eq:3.4} and \eqref{eq:3.6} we obtain that
\begin{align*}
\| s\Phi_{1,2s}^3-\Phi_{1,2}^3\|_{L^2} &\leq (1-s) \| \Phi_{1,2s}^3\|_{L^2} +\| \Phi_{1,2s}^3-\Phi_{1,2}^3\|_{L^2} \\
&\leq (1-s)\| \Phi_{1,2}\|_{L^{\infty}}^2\| \Phi_{1,2}\|_{L^2} 
+3\| \Phi_{1,2}\|_{L^{\infty}}^2\| \Phi_{1,2s}-\Phi_{1,2}\|_{L^2}\\
&\hspace{-5pt}\underset{s\to 1-0}{\longrightarrow} 0.
\end{align*}
Similarly, we have
\begin{align*}
\| \Phi_{1,2s}^5-\Phi_{1,2}^5 \|_{L^2} &\leq 4\|\Phi_{1,2}\|_{L^{\infty}}^4\| \Phi_{1,2s}-\Phi_{1,2}\|_{L^2}
\underset{s\to 1-0}{\longrightarrow} 0.
\end{align*}
Therefore, by using the equation (\ref{eq:3.5}), we deduce that
\begin{align*}
\| \Phi_{1,2s}''-\Phi_{1,2}''\|_{L^2}&\leq (1-s^2)\| \Phi_{1,2s} \|_{L^2}^2+\| s\Phi_{1,2s}^3-\Phi_{1,2}^3\|_{L^2}\\
&\qquad +\frac{3}{16}\gamma\| \Phi_{1,2s}^5-\Phi_{1,2}^5 \|_{L^2}
\underset{s\to 1-0}{\longrightarrow} 0.
\end{align*}
Combined with \eqref{eq:3.4} we have 
\begin{align*}
\lim_{s\to 1-0} \| \Phi_{1,2s} -\Phi_{1,2}\|_{H^2} =0.
\end{align*}
By using the formula \eqref{eq:1.6}, we deduce that
\begin{align*}
\lim_{s\to 1-0} \| \phi_{1,2s} -\phi_{1,2}\|_{H^2} =0.
\end{align*}
The rest of the proof is done by using the equation \eqref{eq:1.5} and a standard bootstrap argument.
\end{proof}
\section{Stability of two types of solitons}
\label{sec:4}
In this section we study stability of two types of solitons for the case $-\frac{3}{16}<b<0$, and prove Theorem \ref{thm:1.2}.

\subsection{Variational characterization}
\label{sec:4.1}
In this subsection we recall variational properties of the solitons of \eqref{ME}. Here we assume that $b$ and $(\omega ,c)$ satisfy
\begin{align}
\label{eq:4.1}
b>-\frac{3}{16}\l(\Leftrightarrow\gamma >0 \r),~ -2\sqrt{\omega} <c\leq 2\sqrt{\omega}.
\end{align}
First we define the function space by
\begin{align*}
\varphi \in X_{\omega,c} &\iff
\l\{
\begin{array}{ll}
\ds \varphi \in H^1(\R) & \ds\text{if}~ \omega >c^2/4,
\\[5pt]
\ds e^{-\frac{i}{2}cx}\varphi \in \dot{H}^1(\R) \cap L^{4}(\R) & \ds\text{if}~c=2\sqrt{\omega},
\end{array}
\r.
\end{align*}
where the norm of $X_{c^2/4,c}$ is defined by
\begin{align*}
\| \varphi\|_{ X_{c^2/4,c} }&=\| e^{-\frac{i}{2}c\cdot}\varphi\|_{\dot{H}^1\cap L^4}.
\end{align*}
We note that $H^1(\R) \subset X_{c^2/4,c}$. We define the functional $\cK_{\omega ,c}$ by
\begin{align*}
\cK_{\omega ,c}(\varphi )&=\l.\frac{d}{d\lambda}\cS_{\omega ,c}(\lambda\varphi)\r|_{\lambda=1}
\\
&= \| \del_x\varphi\|_{L^2}^2+\omega\|\varphi\|_{L^2}^2
+c\rbra[i\del_x\varphi, \varphi] +\frac{c}{2}\| \varphi\|_{L^4}^4 -\frac{3}{16}\gamma\| \varphi\|_{L^6}^6.
\end{align*}
We note that the quadratic terms in the functional are rewritten as 
\begin{align*}
 \| \del_x\varphi\|_{L^2}^2+\omega\|\varphi\|_{L^2}^2
+c\rbra[i\del_x\varphi, \varphi] 
=\norm[ \del_x\l(e^{-\frac{i}{2}c\cdot}\varphi\r)]_{L^2}^2+\l(\omega-\frac{c^2}{4}\r)\|\varphi\|_{L^2}^2.
\end{align*}
From this relation, the functionals $\cS_{\omega,c}$ and $\cK_{\omega,c}$ are well-defined in the space $X_{\omega,c}$. A similar observation is already done in \cite{FHI17, H19}.

Now we consider the following minimization problem:
\begin{align*}
\mu (\omega ,c) &=\inf \l\{ \cS_{\omega ,c}(\varphi) :\varphi \in X_{\omega ,c}\setminus \{ 0\} 
, \cK_{\omega ,c}(\varphi )=0  \r\},
\\
\scM_{\omega ,c}&=\l\{ \varphi \in X_{\omega ,c}\setminus\{ 0\} :
\cS_{\omega ,c}(\varphi) =\mu (\omega ,c), \cK_{\omega ,c}(\varphi )=0  \r\}.
\end{align*}
We note that $\scM_{\omega ,c}$ is the set of minimizers of $\cS_{\omega ,c}$ on the Nehari manifold. The following result gives a variational characterization of the solitons on the Nehari manifold.
\begin{proposition}[\cite{H19}]
\label{prop:4.1}
Assume \eqref{eq:4.1}. Then we have
\begin{align*}
\scM_{\omega ,c} =\l\{ e^{i\theta_0}\varphi_{\omega ,c}(\cdot -x_0):\theta_0\in[0,2\pi), x_0\in\R \r\},
\end{align*}
and $d(\omega ,c)=\mu (\omega ,c)$, where $d(\omega ,c)=\cS_{\omega ,c}(\varphi_{\omega ,c})$ \textup{(}see \eqref{eq:2.4}\textup{)}.
\end{proposition}
Here we introduce the following potential wells in the energy space: 
\begin{align*}
\scA_{\omega ,c}^{+}&=\l\{ v\in H^1(\R)\setminus\{0\} : \cS_{\omega ,c}(v) <d(\omega ,c), 
\cK_{\omega ,c}(v)>0\r\},\\
\scB_{\omega ,c}^{+}&=\l\{ v\in H^1(\R)\setminus\{0\} : \cS_{\omega ,c}(v) <d(\omega ,c), 
\cJ_c (v)<d(\omega ,c)\r\},\\
 \scA_{\omega ,c}^{-}&=\l\{ v\in H^1(\R)\setminus\{0\} : \cS_{\omega ,c}(v) <d(\omega ,c), 
\cK_{\omega ,c}(v)<0\r\},\\
 \scB_{\omega ,c}^{-}&=\l\{ v\in H^1(\R)\setminus\{0\} : \cS_{\omega ,c}(v) <d(\omega ,c), 
\cJ_c (v) >d(\omega ,c)\r\},
\end{align*}
where the functional $\cJ_c$ is defined by
\begin{align*}
\cJ_c(v)=-\frac{c}{8}\| v\|_{L^4}^4+\frac{\gamma}{16}\| v\|_{L^6}^6.
\end{align*}
We note that the functional $\cS_{\omega ,c}$ is rewritten as
\begin{align}
\label{eq:4.2}
\cS_{\omega,c}(v) =\frac{1}{2}\cK_{\omega ,c}(v)+\cJ_c(v).
\end{align}
From Proposition \ref{prop:4.1} we obtain the following result. 
\begin{proposition}
\label{prop:4.2}
Assume \eqref{eq:4.1}. Then
$\scA_{\omega ,c}^{+}$ and $\scA_{\omega ,c}^{-}$ are invariant under the flow of \eqref{ME}. Moreover, we have $\scA_{\omega ,c}^{\pm}=\scB_{\omega ,c}^{\pm}$. 
\end{proposition}
\begin{proof}
The proof is done in the similar way as \cite[Lemma 11]{CO06}.
\end{proof}
\begin{remark*}
One can also prove that if the initial data of \eqref{ME} belongs to $\scA_{\omega ,c}^+$, then the corresponding $H^1(\R)$-solution is global and bounded.
\end{remark*}
Finally we prepare the compactness result on minimizers on the Nehari manifold which is important for the proof of stability. 
\begin{proposition}[\cite{H19}]
\label{prop:4.3}
Assume \eqref{eq:4.1}. If a sequence $
\thickmuskip=3mu \{\varphi_n\}\subset X_{\omega,c}$ satisfies
\begin{align*}
\cS_{\omega ,c}(\varphi_n ) \to\mu (\omega ,c)~\text{and}~\cK_{\omega ,c}(\varphi_n ) \to 0~\text{as}~n\to\infty ,
\end{align*}
then there exist a sequence $\{ y_n\}\subset\R$ and $v\in\scM_{\omega ,c}$ such that $\{ \varphi_n(\cdot -y_n)\}$ has a subsequence that converges to $v$ strongly in $X_{\omega ,c}$.
\end{proposition}


\subsection{Stability theory with potential wells}
\label{sec:4.2}

Here we assume that $b$ and $(\omega ,c)$ satisfy
\begin{align}
\label{eq:4.3}
-\frac{3}{16}< b<0\l(\Leftrightarrow0<\gamma<1\r),~ -2\sqrt{\omega} <c\leq 2\sqrt{\omega}.
\end{align}
We note that $\cP(\varphi_{\omega ,c})>0$ by Proposition \ref{prop:2.4}. To prove stability of the soliton, we need to control the flow around the soliton. By taking advantage of potential wells, we obtain the following claim which plays a key role for the proof of stability.
\begin{proposition}
\label{prop:4.4}
Assume \eqref{eq:4.3}. Then, there exists $\eps_0>0$ such that for any $\eps\in(0,\eps_0)$ there exists $\delta>0$\footnote{As can be seen in the proof, $\delta$ depends on $\eps$ and $(\omega,c)$.} such that if $v_0\in H^1(\R)$ satisfies $\|v_0-\varphi_{\omega ,c}\|_{H^1}<\delta$, then the solution $v(t)$ of \eqref{ME} with $v(0)=v_0$ exists globally in time and satisfies that 
\begin{enumerate}[\rm(i)]
\item if $c =2s\mu$ for $s\in (0,1]$ $(\mu =\sqrt{\omega})$,
\begin{align}
\label{eq:4.4}
\begin{aligned}
d\bigl( (\mu -\eps )^2 ,2s(\mu -\eps )\bigr)& -\frac{s\eps}{4}\| v(t)\|_{L^4}^4 \\
<\cJ_{c} (v(t)) &<d\bigl( (\mu +\eps )^2 ,2s(\mu +\eps )\bigr) +\frac{s\eps}{4}\| v(t)\|_{L^4}^4,
\end{aligned}
\end{align}

\item if $c=0$,
\begin{align}
\label{eq:4.5}
d(\omega ,-\eps )-\frac{\eps}{8}\| v(t)\|_{L^4}^4 <\cJ_0 (v(t)) <d(\omega , \eps )+\frac{\eps}{8}\| v(t)\|_{L^4}^4 ,
\end{align}

\item if $c<0$,
\begin{align}
\label{eq:4.6}
d(\omega -\eps ,c ) <\cJ_{c}(v(t)) <d(\omega +\eps ,c),
\end{align}
\end{enumerate}
for all $t\in \R$ in {\rm (i)-(iii)}.
\end{proposition}
\begin{remark*}
Compared with the corresponding result \cite[Lemma 12]{CO06}, the $L^4$-norm appears in \eqref{eq:4.4} and \eqref{eq:4.5}, which comes from the lack of the ``good" Hamiltonian structure in \eqref{ME}.
\end{remark*}
\begin{proof}
We mainly prove the most difficult case $0<c\leq 2\sqrt{\omega}$. 

(i) Let $\eps_0>0$ be sufficiently small. For $\eps\in (0,\eps_0)$ we define the function $g$ by
\begin{align*}
g(\tau )=d\bigl( (\mu +\tau)^2 ,2s(\mu +\tau)\bigr) \quad\text{for}~\tau\in (-\eps ,\eps ).
\end{align*}
From the relation \eqref{eq:2.5} we have
\begin{align*}
&g(\tau )=(\mu +\tau )^2d(1,2s)\quad\text{for}~\tau\in (-\eps ,\eps ),
\end{align*}
which yields that
\begin{align}
\label{eq:4.7}
g(0)=\mu^2d(1,2s),~g'(0)=2\mu d(1,2s), ~g''(0)=2d(1,2s).
\end{align}
From Lemma \ref{lem:2.5} we have
\begin{align}
\label{eq:4.8}
2d(1,2s)=\cM (\varphi_{1,2s})+s\cP (\varphi_{1,2s}).
\end{align}
Assume that $v_0\in H^1(\R)$ satisfies $\| v_0-\varphi_{\mu^2,2s\mu}\|_{H^1}<\delta$, where $\delta>0$ is determined later. First we prove that 
\begin{align}
\label{eq:4.9}
v_0 \in \scB^+_{({\mu +\eps})^2,2s({\mu +\eps})}
\cap \scB^-_{({\mu -\eps})^2,2s({\mu -\eps})}.
\end{align}
From \eqref{eq:4.8} and \eqref{eq:4.7} we have
\begin{align*}
\cS_{(\mu \pm\eps )^2,2s(\mu\pm\eps )}(v_0)
&=\cS_{(\mu \pm\eps )^2,2s(\mu\pm\eps )}(\varphi_{\mu^2,2s\mu})+O(\delta )\\
&=\cE (\varphi_{\mu^2,2s\mu})+\frac{(\mu\pm\eps )^2}{2}\cM (\varphi_{\mu^2,2s\mu})\\
&\qquad\qquad\qquad +s(\mu\pm\eps )\cP (\varphi_{\mu^2,2s\mu})+O(\delta )\\
&=\mu^2 d(1,2s)\pm\eps\mu\l( \cM (\varphi_{1,2s})+s\cP (\varphi_{1,2s})\r)\\
&\qquad\qquad\qquad +\frac{\eps^2}{2}\cM (\varphi_{1,2s})+O(\delta )\\
&=g(0)\pm\eps g'(0)+\frac{\eps^2}{2}\cM (\varphi_{1,2s})+O(\delta ).
\end{align*}
By using the Taylor expansion,\footnote{One can also show this formula without using the Taylor expansion since the function $g$ is the quadratic function.} we have
\begin{align*}
g(\pm\eps )&=g(0)\pm\eps g'(0)+\frac{\eps^2}{2}g''(0 ).
\end{align*}
We note that
\begin{align*}
g''(0)=2d(1,2s)=\cM (\varphi_{1,2s})+s\cP (\varphi_{1,2s})
\end{align*}
and $s\cP(\varphi_{1,2s})>0$. Therefore, by taking small $\delta >0$ we obtain that
\begin{align}
\label{eq:4.10}
\cS_{(\mu \pm\eps )^2,2s(\mu\pm\eps )}(v_0) < g(\pm\eps ).
\end{align}
On the other hand, by (\ref{eq:4.2}) and $\cK_{\omega ,c}(\varphi_{\omega,c})=0$ we have
\begin{align*}
\cJ_{c+2s\eps}(\varphi_{\omega ,c})&=
-\frac{c+2s\eps}{8}\|\varphi_{\omega ,c}\|_{L^4}^4 +\frac{\gamma}{16}\| \varphi_{\omega ,c}\|_{L^6}^6 \\
&<\cJ_{c} (\varphi_{\omega ,c})=g(0) <g(\eps ).
\end{align*}
By taking smaller $\delta>0$ again, we obtain that $\cJ_{c+2s\eps} (v_0)<g(\eps )$. Similarly, we have 
$g(-\eps )<\cJ_{c-2s\eps} (v_0)$.
Combined with \eqref{eq:4.10}, we deduce that \eqref{eq:4.9} holds. 

We now prove \eqref{eq:4.4}. By Proposition \ref{prop:4.2} we have
\begin{align}
\label{eq:4.11}
v(t) \in \scB^+_{({\mu +\eps})^2,2s({\mu +\eps})}
\cap \scB^-_{({\mu -\eps})^2,2s({\mu -\eps})}
\end{align}
for all $t\in \R$. Therefore, we deduce that
\begin{align*}
g(\eps )>\cJ_{c+2s\eps} (v(t))
&=-\frac{c+2s\eps}{8}\| v(t)\|_{L^4}^4 +\frac{\gamma}{16}\| v(t)\|_{L^6}^6
\\
&=\cJ_{c}(v(t))-\frac{s\eps}{4} \| v(t)\|_{L^4}^4.
\end{align*}
Similarly, we have
\begin{align*}
g(-\eps )<\cJ_{c}(v(t))+\frac{s\eps}{4} \| v(t)\|_{L^4}^4.
\end{align*}
This completes the proof of \eqref{eq:4.4}.
\\[5pt]
(ii) When $c=0$, by Lemma \ref{lem:2.3} we have
\begin{align*}
\Bigl. \del_c \cP(\varphi_{\omega ,c})\Bigr|_{c=0}= \frac{1}{2} \l( -1+ \frac{1}{\gamma} \r) M(\varphi_{\omega ,0})>0,
\end{align*}
which yields that $\Bigl.\del_c^2 d(\omega ,c)\Bigr|_{c=0}>0$. From this fact and the calculation based on the function
$
(-\eps,\eps)\ni\tau\mapsto d(\omega,\tau),
$
one can prove that 
\begin{align}
\label{eq:4.12}
v_0 \in \scB^+_{\omega ,\eps}
\cap \scB^-_{\omega ,-\eps}.
\end{align}
In the same way as (i), we see that \eqref{eq:4.12} implies \eqref{eq:4.5}.
\\[5pt]
(iii) When $c<0$, by Lemma \ref{lem:2.2} we have 
\begin{align*}
\del^2_{\omega} d(\omega ,c)=\frac{1}{2}\del_{\omega}\cM (\varphi_{\omega ,c}) >0.
\end{align*}
From this fact and the calculation based on the function
$
(-\eps,\eps)\ni\tau\mapsto d(\omega+\tau,c),
$
one can prove that 
\begin{align*}
v_0 \in \scB^+_{\omega+\eps ,c}
\cap \scB^-_{\omega-\eps ,c},
\end{align*}
which yields \eqref{eq:4.6}.
\end{proof}
Combined with Proposition \ref{prop:4.3}, one can prove the following stability result.
\begin{theorem}
\label{thm:4.5}
Assume \eqref{eq:4.3}. Then the soliton $v_{\omega,c}$ of \eqref{ME} is stable.
\end{theorem}
\begin{proof}
The claim is proved by contradiction. Assume that there exist $\eps_1 >0$, a sequence of the maximal solutions $\{v_n\}$ to \eqref{ME} and a sequence $\{ t_n\}\subset\R$ such that 
\begin{align}
\label{eq:4.13}
\| v_n(0)-\varphi_{\omega,c} \|_{H^1} \underset{n\to\infty}{\longrightarrow} 0,\\
\label{eq:4.14}
\inf_{(\theta ,y)\in\R^2}\| v_n(t_n)-e^{i\theta}\varphi_{\omega ,c}(\cdot -y)\|_{H^1}\geq\eps_1 .
\end{align}
Since $\cS_{\omega ,c}(\cdot)$ is a conserved quantity, by \eqref{eq:4.13} we have
\begin{align}
\label{eq:4.15}
\cS_{\omega ,c}(v_n (t_n))=\cS_{\omega ,c}(v_n(0))
\underset{n\to\infty}{\longrightarrow}\cS_{\omega ,c}(\varphi_{\omega ,c}) =d(\omega ,c).
\end{align}
By \eqref{eq:4.13}, \eqref{eq:4.14} and the continuity $t\mapsto v (t)\in H^1(\R)$, one can pick up $t_n$ (still denoted by the same letter) such that
\begin{align}
\label{eq:4.16}
\inf_{(\theta ,y)\in\R^2}\| v_n(t_n)-e^{i\theta}\varphi_{\omega ,c}(\cdot -y)\|_{H^1}=\eps_1 .
\end{align}
This equality yields the boundedness of $\{ v_n(t_n)\}$ in $H^1(\R)$, i.e.,
\begin{align}
\label{eq:4.17}
\sup_{n\in\N} \| v_n(t_n)\|_{H^1} \leq C,
\end{align}
where $C$ only depends on $\| \varphi_{\omega ,c}\|_{H^1}$ and $\eps_1$. 
From Proposition \ref{prop:4.4} and \eqref{eq:4.17}, we obtain that 
\begin{align*}
\cJ_{c}(v_n(t_n))\underset{n\to\infty}{\longrightarrow} d(\omega ,c).
\end{align*}
Combined with \eqref{eq:4.2}, we have
\begin{align}
\label{eq:4.18}
\cK_{\omega ,c}(v_n(t_n))
\underset{n\to\infty}{\longrightarrow} 0.
\end{align}
Therefore, by (\ref{eq:4.15}), (\ref{eq:4.18}) and Proposition \ref{prop:4.3}, there exist a sequence $\{ y_n\}$ and $\theta_0 ,y_0 \in\R$ such that $\{ v_n(t_n, \cdot +y_n)\}$ has a subsequence (still denoted by the same letter) that converges to $e^{i\theta_0}\varphi_{\omega,c}(\cdot -y_0)$ in $X_{\omega ,c}$. If $\omega >c^2/4$, this yields that
\begin{align}
\label{eq:4.19}
\| v_n(t_n)-e^{i\theta_0}\varphi_{\omega ,c}(\cdot -y_0-y_n)\|_{H^1}\underset{n\to\infty}{\longrightarrow} 0,
\end{align} 
which contradicts \eqref{eq:4.16}.

When $c =2\sqrt{\omega}$, we need to modify the argument slightly. From the definition of $X_{c^2/4,c}$, we have
\begin{align}
\label{eq:4.20}
e^{-\frac{i}{2}c\cdot}v_n(t_n, \cdot +y_n) \to
e^{-\frac{i}{2}c\cdot}e^{i\theta_0}\varphi_{\omega,c}(\cdot -y_0)~\text{in}~\dot{H}^1(\R) .
\end{align}
By using this convergence one can easily prove that 
\begin{align}
\label{eq:4.21}
e^{-\frac{i}{2}c\cdot }v_n(t_n, \cdot +y_n) 
\wto
e^{-\frac{i}{2}c\cdot }e^{i\theta_0}\varphi_{\omega,c}(\cdot -y_0)~\text{weakly in}~L^2(\R) .
\end{align}
From \eqref{eq:4.13} and mass conservation we obtain that
\begin{align}
\label{eq:4.22}
\cM (v_n (t_n)) =\cM (v_n(0) ) \to \cM (\varphi_{\omega ,c}).
\end{align}
Therefore, it follows from \eqref{eq:4.21} and \eqref{eq:4.22} that
\begin{align}
\label{eq:4.23}
e^{-\frac{i}{2}c\cdot }v_n(t_n, \cdot +y_n) \to
e^{-\frac{i}{2}c\cdot}e^{i\theta_0}\varphi_{\omega,c}(\cdot -y_0)~\text{in}~L^2(\R) .
\end{align}
Hence \eqref{eq:4.19} follows from \eqref{eq:4.20} and \eqref{eq:4.23}, which contradicts \eqref{eq:4.16}. This completes the proof.
\end{proof}
\begin{proof}[Proof of Theorem \ref{thm:1.2}]
We note that $v_{\omega,c}=\cG(u_{\omega,c})$ and 
\begin{align*}
\cG(e^{i\theta}u(\cdot-y))(x)=e^{i\theta}\cG(u)(x -y)
\end{align*}
for $u\in H^1(\R)$ and $x,y,\theta\in \R$. We also note that the gauge transformation $u\mapsto \cG(u)$ is Lipschitz continuous on bounded subsets of $H^1(\R)$. Hence the result follows from Theorem \ref{thm:4.5} and the properties of the gauge transformation.
\end{proof}
\begin{remark*}
The stability of the solitons for the case $b=-\frac{3}{16}$ is proved in the same way. Indeed, the results in Section \ref{sec:4.1} still hold in this case, and Proposition \ref{prop:4.4} (iii) holds since velocity of the solitons is negative. 
\end{remark*}

We note that the formula \eqref{eq:1.15} still holds including the case $b<0$, i.e.,
\begin{align}
\label{eq:4.24}
\mathrm{det}[d''(\omega ,c)]=\frac{-2P(\phi_{\omega ,c})}{\sqrt{4\omega -c^2}
\l\{c^2+\gamma (4\omega -c^2) \r\} } 
\quad\text{for}~\omega>c^2/4.
\end{align}
By Proposition \ref{prop:2.4}, the momentum $P(\phi_{\omega,c})$ is always positive when $b<0$, which yields that $d''(\omega ,c)$ has one positive eigenvalue. Therefore there exists $\xi\in\R$ such that 
\begin{align*}
\tbra[d'(\omega ,c), \xi]\neq 0,~\tbra[d''(\omega ,c)\xi ,\xi]>0.
\end{align*}
As in the proof of Proposition \ref{prop:4.4}, the calculation of the function $\tau\mapsto d( (\omega,c)+\tau\xi)$ and variational characterization yields the control of the flow around the soliton. This is an adaptation of the argument in \cite{CO06} to our setting, but one cannot treat algebraic solitons in this approach.

Our variational approach offers a new perspective to the stability theory of a two-parameter family of solitons. We note that Proposition \ref{prop:4.4} is obtained without calculating the Hessian matrix $d''(\omega ,c)$. The calculation along the scaling curve gives a simpler argument on the stability theory, and also enables us to treat algebraic solitons and exponentially decaying solitons in a unified way.
This indicates that the curve \eqref{eq:1.10} gives not only the scaling of the soliton but also ``good" measure of the stability.


\section{Stability of solitons with negative velocity}
\label{sec:5}
In this section we study stability of the solitons for the case $b\le -\frac{3}{16}$, and prove Theorem \ref{thm:1.4}. For the proof we apply variational arguments introduced by Cazenave and Lions \cite{CL82}. Here we assume that $b$ and $(\omega,c)$ satisfy
\begin{align}
\label{eq:5.1}
b\le-\frac{3}{16}\l(\Leftrightarrow \gamma \le0\r),~ -2\sqrt{\omega} <c< -2s_*\sqrt{\omega}.
\end{align}
We remark that our proof in this section still works for the case $b>-\frac{3}{16}$ and 
$ -2\sqrt{\omega} <c<0$ (see the end of this section).

First we note that 
\begin{align}
\label{eq:5.2}
\begin{aligned}
\cS_{\omega ,c}(e^{\frac{i}{2}cx}\psi)
=& \frac{1}{2}\l\| \del_x\psi \r\|_{L^2}^2
 +\frac{1}{2}\l(\omega-\frac{c^2}{4}\r)\| \psi\|_{L^2}^2 +\frac{c}{8}\| \psi\|_{L^4}^4 -\frac{\gamma}{32}\| \psi\|_{L^6}^{6}
 \\
=& \cE_{c}(\psi) +\frac{1}{2}\l(\omega-\frac{c^2}{4}\r)\| \psi\|_{L^2}^2,
\end{aligned}
\end{align}
where $\cE_c$ is defined by
\begin{align*}
\cE_{c}(\psi)= \frac{1}{2}\l\| \del_x\psi \r\|_{L^2}^2+\frac{c}{8}\| \psi\|_{L^4}^4 -\frac{\gamma}{32}\| \psi\|_{L^6}^{6}.
\end{align*}
We note that $\cS_{\omega ,c}'(e^{\frac{i}{2}cx}\psi)=0$ is equivalent that
\begin{align*}
-\psi'' +\l( \omega -\frac{c^2}{4}\r)\psi +\frac{c}{2}|\psi|^2\psi -\frac{3}{16}\gamma|\psi|^4\psi=0,
\quad x\in\R,
\end{align*}
which is nothing but \eqref{eq:1.7}.

Now we consider a variational problem with mass constraint:
\begin{align*}
\scA_m =& \l\{ \psi\in H^1(\R) : \| \psi\|_{L^2}^2 =m \r\},\\
-\nu (c,m)&=\inf\l\{ \cE_c(\psi) : \psi\in \scA_m \r\},\\
\scM_{c,m}=&\l\{ \psi\in\scA_m : \cE_c(\psi ) =-\nu(c,m)\r\}
\end{align*}
for $m>0$. We begin with the following lemma.
\begin{lemma}
\label{lem:5.1}
Assume $\gamma\le 0$, $c<0$ and $m>0$. Then $-\infty<-\nu (c,m)<0$.
\end{lemma}
\begin{proof}
From the assumption, $\cE_c$ is rewritten as
\begin{align*}
\cE_c(\psi) =\frac{1}{2}\l\| \del_x\psi \r\|_{L^2}^2-\frac{|c|}{8}\| \psi\|_{L^4}^4 +\frac{|\gamma|}{32}\| \psi\|_{L^6}^{6}.
\end{align*}
For $\psi\in\scA_m$ we set $\psi_{\lambda}=\lambda^{1/2}\psi(\lambda x)$. 
Then, $\psi_{\lambda}\in\scA_m$ and
\begin{align*}
\cE_c (\psi_{\lambda}) &=\lambda^2\l( \frac{1}{2}\| \del_x\psi \|_{L^2}^2 +\frac{|\gamma|}{32}\| \psi\|_{L^6}^6 \r)-\frac{|c|}{8}\lambda\| \psi\|_{L^4}^4
\\
&=\lambda^2\l( \cE(\psi) -\lambda^{-1}\frac{|c|}{8}\| \psi\|_{L^4}^4\r).
\end{align*}
One can see that $\cE_c(\psi_{\lambda} )<0$ for sufficiently small $\lambda>0$, which yields that $-\nu (c,m)<0$. 

By using the Gagliardo--Nirenberg's inequality
\begin{align*}
\|f\|_{L^4}\leq C_1 \| \del_xf\|_{L^2}^{1/4}\| f\|_{L^2}^{3/4},
\end{align*}
we obtain that
\begin{align}
\label{eq:5.3}
 \begin{aligned}
\cE_c(\psi)
&\ge \frac{1}{2}\l\| \del_x\psi \r\|_{L^2}^2 -C_1\frac{|c|}{8}\| \del_x\psi\|_{L^2}\| f\|_{L^2}^{3}
\ge \frac{1}{4}\l\| \del_x\psi \r\|_{L^2}^2 -C_2c^2\| \psi\|_{L^2}^6
\end{aligned}
\end{align}
for some constant $C_2>0$. 
Therefore we deduce that
\begin{align*}
-\nu (c,m) =\inf_{\psi\in\scA_m}\cE_c(\psi) \ge -C_2c^2m^3>-\infty.
\end{align*}
This completes the proof.
\end{proof}
The following claim on the sequential compactness plays a key role for the proof of stability in this section.
\begin{proposition}
\label{prop:5.2}
Assume $\gamma\le 0$, $c<0$ and $m>0$. If a sequence $\{\psi_n\}\subset H^1(\R)\setminus\{0\}$ satisfies $\| \psi_n\|_{L^2}^2\to m$ and $\cE_c(\psi_n)\to -\nu(c,m)$, then there exist $\varphi\in\scM_{c,m}$  and a sequence $\{y_n\}\subset\R$, such that $\{\psi_{n}(\cdot-y_n)\}$ has a subsequence that converges to $\varphi$ strongly in $H^1(\R)$.
\end{proposition}
For the proof of Proposition \ref{prop:5.2} we use the following Lieb's compactness lemma and Brezis--Lieb's lemma (Lemma \ref{lem:3.2}). We note that the original argument in \cite{CL82} (see also \cite[Chapter 8]{C03}) relies on the concentration compactness method by Lions \cite{Lions84}.  
\begin{lemma}[\cite{L83}]
\label{lem:5.3}
Let $\{f_n\}$ be a bounded sequence in $H^1(\R)$. Assume that there exists $q\in(2,\infty)$ such that $\limsup_{n \to \infty} \norm[f_n]_{L^q}>0$.
Then, there exist $\{y_n\}\subset\R$ and $f \in H^1(\R)\setminus \{0\}$ such that $\{f_n(\cdot-y_n)\}$ has a subsequence that converges to $f$ weakly in $H^1(\R)$. 
\end{lemma}
\begin{proof}[Proof of Proposition \ref{prop:5.2}]
We proceed in three steps.
\\[5pt]
{\bf Step 1:} Boundedness of $\{\psi_n\}$. From the assumption and $-\nu(c,m)<0$, $\cE_c(\psi_n)<-\nu(c,m)/2$ for large $n$.
Combined with \eqref{eq:5.3}, we obtain that
\begin{align*}
-\frac{\nu(c,m)}{2} >\cE_c(\psi_n) \ge \frac{1}{4}\| \del_x\psi_n\|_{L^2}^2 -C_2c^2m^3
\end{align*}
for large $n$. Since $\|\psi_n\|_{L^2}^2\to m$, this yields that $\{ \psi_n\}$ is bounded in $H^1(\R)$. From the definition of $\cE_c$ and $\cE\ge 0$, we have
\begin{align*}
-\frac{\nu(c,m)}{2} >\cE_c (\psi_n) =\cE(\psi_n) -\frac{|c|}{8}\| \psi_n\|_{L^4}^4
\ge -\frac{|c|}{8}\| \psi_n\|_{L^4}^4,
\end{align*}
which implies that 
\begin{align*}
0< \frac{4\nu(c,m)}{|c|}<\| \psi_n\|_{L^4}^4 \quad\text{for large}~n. 
\end{align*}
Thus we have obtained that
\begin{align}
\label{eq:5.4}
\sup_{n\in\N}\|\psi_n\|_{H^1}<\infty,~\inf_{n\in\N}\| \psi_n\|_{L^4}^4>0.
\end{align}
\noindent
{\bf Step 2:} Limits. From \eqref{eq:5.4} one can apply Lemma \ref{lem:5.3} to the sequence $\{ \psi_n\}$. Then there exist $\{y_n\}\subset\R$ and $\varphi\in H^1(\R)\setminus\{ 0\}$ such that a subsequence of $\{ \psi_n(\cdot -y_n)\}$ (we denote it by $\{ \varphi_n\}$) converges to $\varphi$ weakly in $H^1(\R)$.  
By the weak lower semicontinuity of the $L^2$ norm we have
\begin{align}
\label{eq:5.5}
\|\varphi\|_{L^2}^2\le\liminf_{n\to\infty} \| \varphi_n\|_{L^2}^2=\lim_{n\to\infty}\|\psi_n\|_{L^2}^2=m.
\end{align}
We also have, up to a subsequence, $\varphi_n\to\varphi~\text{a.e.}~\text{in}~\R$. Applying Lemma \ref{lem:3.2} we obtain that
\begin{align}
\label{eq:5.6}
&\cE_c(\varphi_n)-\cE_c(\varphi_n-\varphi)-\cE_c (\varphi)\to 0,\\
\label{eq:5.7}
&\| \varphi_n \|_{L^2}^2 -\| \varphi_n-\varphi\|_{L^2}^2 -\| \varphi\|_{L^2}^2\to 0.
\end{align}
\noindent
{\bf Step 3:} Strong convergence. 
Assume that $\|\varphi\|_{L^2}^2<m$. Then, combined with \eqref{eq:5.7} and $\varphi\neq 0$, we have
\begin{align*}
0<\lim_{n\to\infty}\| \varphi_n-\varphi\|_{L^2}^2 <m.
\end{align*}
We set $\xi_n =\varphi_n-\varphi$. Following an idea from \cite{CJS10, CS20}, we modify $\{ \xi_n\}$ and $\varphi$ by using the scaling transformation
\begin{align*}
\til{\xi}_n(x)=\xi_n(\lambda_n^{-1}x),~\til{\varphi}(x)=\varphi(\lambda^{-1}x ),
\end{align*} 
where
\begin{align*}
\lambda_n=\frac{m}{\| \xi_n\|_{L^2}^2},~
\lambda =\frac{m}{\| \varphi\|_{L^2}^2}.
\end{align*}
We note that $\lambda, \lambda_n>1$ and $\til{\xi}_n, \til{\varphi}\in\scA_m$.
By a direct calculation we have 
\begin{align}
\label{eq:5.8}
\begin{aligned}
\cE_c(\varphi)&=\frac{1-\lambda^{-2}}{2}\|\del_x\varphi\|_{L^2}^2+\lambda^{-1}\cE_c(\til{\varphi}),
\\
\cE_c(\xi_n)&=\frac{1-\lambda_n^{-2}}{2}\|\del_x\xi_n\|_{L^2}^2+\lambda_n^{-1}\cE_c(\til{\xi}_n).
\end{aligned}
\end{align}
Then it follows from \eqref{eq:5.6}, \eqref{eq:5.8} and \eqref{eq:5.7} that
\begin{align*}
-\nu(c,m) =\lim_{n\to\infty}\cE_c (\varphi_n)
&=\lim_{n\to\infty} \cE_c(\xi_n)+\cE_c (\varphi)
\\
&=
\begin{aligned}[t]
\lim_{n\to\infty} \Bigl[ \frac{1-\lambda_n^{-2}}{2}\|\del_x\xi_n\|_{L^2}^2&+\frac{1-\lambda^{-2}}{2}\|\del_x\varphi\|_{L^2}^2
\\
&{}+\lambda_n^{-1}\cE_c(\til{\xi}_n)+\lambda^{-1}\cE_c(\til{\varphi}) \Bigr]
\end{aligned}
\\
&\ge \frac{1-\lambda^{-2}}{2}\|\del_x\varphi\|_{L^2}^2 -\nu(c,m)\lim_{n\to\infty}\l( \lambda_n^{-1}+\lambda^{-1} \r)
\\
&=\frac{1-\lambda^{-2}}{2}\|\del_x\varphi\|_{L^2}^2 -\nu(c,m)>-\nu(c,m),
\end{align*}
which gives a contradiction. Therefore we deduce that $m=\|\varphi\|_{L^2}^2$. 

Since we have the relation
\begin{align*}
\lim_{n\to\infty}\| \varphi_n\|_{L^2}^2=m=\|\varphi\|_{L^2}^2,
\end{align*}
we deduce that
$
\varphi_n\to\varphi~\text{in}~L^2(\R)
$.
From boundedness of $\{\varphi_n\}$ in $H^1(\R)$ and elementary interpolation estimates, we have
\begin{align}
\label{eq:5.9}
\varphi_n \to \varphi\quad \text{in}~L^r(\R)\quad\text{for all}~r\in[2,\infty].
\end{align}
Combined with the lower semicontinuity of the $H^1$-norm, we deduce that 
\begin{align*}
\cE_c (\varphi)\le\liminf_{n\to\infty}\cE_c(\varphi_n) =-\nu(c,m).
\end{align*}
On the other hand, it follows from $\varphi\in\scA_m$ that $-\nu(c,m)\le\cE_c (\varphi)$, which yields that $-\nu(c,m) =\cE_c(\varphi)$. Hence $\varphi\in\cM_{c,m}$. By \eqref{eq:5.6} we have $\cE(\varphi_n-\varphi)\to 0$. Combined with \eqref{eq:5.9} we obtain that 
\begin{align*}
\frac{1}{2}\|\del_x\varphi_n-\del_x\varphi\|_{L^2}^2
=\cE_c(\varphi_n-\varphi)-\frac{c}{8}\| \varphi_n-\varphi\|_{L^4}^4
+\frac{\gamma}{32}\| \varphi_n-\varphi\|_{L^6}^6
\to 0,
\end{align*}
which yields that
\begin{align*}
\varphi_n \to \varphi\quad\text{strongly in}~H^1(\R).
\end{align*}
This completes the proof.
\end{proof}
The set $\scM_{c,m}$ is characterized as follows.
\begin{lemma}
\label{lem:5.4}
Assume \eqref{eq:5.1}. Suppose further that
\begin{align}
\label{eq:5.10}
m=\| \varphi_{\omega,c}\|_{L^2}^2=\| \Phi_{\omega,c}\|_{L^2}^2.
\end{align}
Then we have 
\begin{align}
\label{eq:5.11}
\scM_{c,m} =\l\{ e^{i\theta}\Phi_{\omega,c}(\cdot-y) :\theta, y\in\R \r\}
~\text{and}~
-\nu(c,m)=\cE_{c}(\Phi_{\omega,c}).
\end{align}
\end{lemma}
\begin{proof}
By Proposition~\ref{prop:5.2} we note that $\cM_{c,m}\neq\emptyset$. Let $\psi\in\scM_{c,m}$. Then there exists a Lagrange multiplier $\lambda\in\R$ such that
\begin{align*}
\cE_{c}'(\psi)+\lambda\cM '(\psi)=0
\iff -\psi'' +\lambda\psi +\frac{c}{2}|\psi|^2\psi-\frac{3}{16}\gamma|\psi|^4\psi=0.
\end{align*}
Since $\psi\neq 0$, one can easily prove that $\lambda>0$. If we set 
$
\tilde{\omega}=\lambda +\frac{c^2}{4}>0,
$
then $\psi$ satisfies the equation
\begin{align}
\label{eq:5.12}
-\psi'' +\l( \tilde{\omega} -\frac{c^2}{4}\r)\psi +\frac{c}{2}|\psi|^2\psi -\frac{3}{16}\gamma|\psi|^4\psi=0,\quad x\in\R.
\end{align}
By uniqueness of the solution of \eqref{eq:5.12}, there exist $\theta,y\in \R$ such that
\begin{align*}
\psi =e^{i\theta}\Phi_{\tilde{\omega},c}(\cdot-y).
\end{align*}
From the assumption we have
\begin{align*}
\| \Phi_{\tilde{\omega},c}\|_{L^2}^2=\| \psi\|_{L^2}^2=m=\| \Phi_{\omega,c}\|_{L^2}^2.
\end{align*}
Since $c<0$, it follows from Lemma~\ref{lem:2.2} that the function
\begin{align*}
 \l(\frac{c^2}{4},\infty \r)\ni \mu\mapsto \| \Phi_{\mu ,c}\|_{L^2}^2\in (0,\infty) 
\end{align*}
is strictly increasing, in particular which implies that $\tilde{\omega}=\omega$. Hence we have $\psi=e^{i\theta}\Phi_{\omega,c}(\cdot-y)$. We also obtain that
\begin{align}
\label{eq:5.13}
-\nu(c,m)=\cE_{c}(\psi)=\cE_{c} (\Phi_{\omega,c}).
\end{align}

Conversely, if $\psi =e^{i\theta}\Phi_{\omega,c}(\cdot-y)$ for some $\theta,y\in \R$, then it follows 
from \eqref{eq:5.10} and \eqref{eq:5.13} that $\psi\in\scM_{c,m}$. This completes the proof.
\end{proof}
Next we prove the following claim on sequential compactness.
\begin{proposition}
\label{prop:5.5}
Assume \eqref{eq:5.1}.
Suppose further that $m$ is defined by \eqref{eq:5.10}. If a sequence $\{\varphi_n\}\subset H^1(\R)$ satisfies
\begin{align*}
\cE(\varphi_n)\to\cE(\varphi_{\omega,c}),~\cP(\varphi_n)\to
\cP(\varphi_{\omega,c}),~\cM(\varphi_n)\to \cM(\varphi_{\omega,c}),
\end{align*}
then there exist a subsequence of $\{\varphi_n\}$ $($still denoted by the same letter$)$ and $\{ \theta_n\}, \{ y_n\}\subset\R$ such that
\begin{align*}
e^{i\theta_n}\varphi_n(\cdot-y_n) \to \varphi_{\omega,c}\quad\text{strongly in}~H^1(\R).
\end{align*} 
\end{proposition}
\begin{proof}
We first note that  
\begin{align}
\label{eq:5.14}
\cE_c(e^{-\frac{i}{2}cx}\psi)=\cS_{\omega ,c}(\psi)-\frac{1}{2}\l(\omega-\frac{c^2}{4}\r)\| \psi\|_{L^2}^2
\quad\text{for}~\psi\in H^1(\R),
\end{align}
which follows from \eqref{eq:5.2}. If we set $\varphi=\varphi_{\omega ,c}$, we have
\begin{align}
\label{eq:5.15}
\cE_c(e^{-\frac{i}{2}cx}\varphi_{\omega,c})=d(\omega ,c)
-\frac{1}{2}\l(\omega-\frac{c^2}{4}\r)\| \varphi_{\omega ,c}\|_{L^2}^2.
\end{align}
From the assumption we have
\begin{align*}
\cS_{\omega,c}(\varphi_n)\to \cS_{\omega,c}(\varphi_{\omega,c})=d(\omega,c).
\end{align*}
Combined with \eqref{eq:5.14} and \eqref{eq:5.15}, we have
\begin{align*}
\cE_{c}(e^{-\frac{i}{2}cx}\varphi_n) \to \cE_c(e^{-\frac{i}{2}cx}\varphi_{\omega,c})=\cE_c(\Phi_{\omega ,c})
=-\nu (c,m),
\end{align*}
where we used \eqref{eq:2.2} and \eqref{eq:5.11}.
Therefore, by Proposition~\ref{prop:5.2} and Lemma~\ref{lem:5.4}, there exist $\{z_n\}\subset\R$ and $\theta,y\in\R$ such that up to a subsequence,
\begin{align*}
e^{-\frac{i}{2}c(\cdot-z_n)}\varphi_n(\cdot -z_n)\to e^{i\theta}\Phi_{\omega,c}(\cdot -y) \quad\text{strongly in}~H^1(\R).
\end{align*}
This yields that
\begin{align*}
e^{-\frac{i}{2}c(y-z_n)-i\theta}\varphi_n(\cdot+y -z_n)\to e^{\frac{i}{2}c\cdot}\Phi_{\omega,c}=\varphi_{\omega,c} \quad\text{strongly in}~H^1(\R),
\end{align*}
which completes the proof.
\end{proof}
We are now in a position to prove the following stability result.
\begin{theorem}
\label{thm:5.6}
Assume \eqref{eq:5.1}. Then the soliton $v_{\omega,c}$ of \eqref{ME} is stable.
\end{theorem}
\begin{proof}
For completeness we give a proof. Assume by contradiction that  there exist $\eps >0$, a sequence of the maximal solutions $\{v_n\}$ to \eqref{ME} and a sequence $\{ t_n\}\subset\R$ such that 
\begin{align}
\label{eq:5.16}
\| v_n(0)-\varphi_{\omega,c} \|_{H^1} \underset{n\to\infty}{\longrightarrow} 0,\\
\label{eq:5.17}
\inf_{(\theta ,y)\in\R^2}\| v_n(t_n)-e^{i\theta}\varphi_{\omega ,c}(\cdot -y)\|_{H^1}\geq\eps .
\end{align}
From conservation laws and \eqref{eq:5.16}, we have
\begin{align*}
\begin{aligned}
\cE(v_n (t_n))=\cE(v_n(0))&\to\cE(\varphi_{\omega ,c}),
\\
\cM(v_n (t_n))=\cM(v_n(0))&\to\cM(\varphi_{\omega ,c}),
\\
\cP(v_n (t_n))=\cP(v_n(0))&\to\cP(\varphi_{\omega ,c}).
\end{aligned}
\end{align*}
Therefore, by Proposition \ref{prop:5.5}, there exist a subsequence of $\{ v_n(t_n)\}$ (still denoted by the same letter) and $\{ \theta_n\}, \{ y_n\}\subset\R$ such that
\begin{align*}
v_n(t_n)-e^{i\theta_n}\varphi_{\omega,c}(\cdot -y_n)\to 0\quad\text{strongly in}~H^1(\R),
\end{align*}
which contradicts \eqref{eq:5.17}.
\end{proof}
\begin{proof}[Proof of Theorem \ref{thm:1.4}]
Similarly as in the proof of Theorem \ref{thm:1.2}, the result follows from Theorem \ref{thm:5.6} and the properties of the gauge transformation $u\mapsto \cG(u)$.
\end{proof}
Our proof in this section still works for the case $b>-\frac{3}{16}$ and 
$ -2\sqrt{\omega} <c<0$. For this case we note that
\begin{align*}
0<\| \varphi_{\omega ,c}\|_{L^2}^2 <\| \varphi_{\omega ,0}\|_{L^2}^2=\frac{2\pi}{\sqrt{\gamma}},
\end{align*}
which follows from Lemma \ref{lem:2.1}. By the sharp Gagliardo--Nirenberg inequality
\begin{align*}
\frac{\gamma}{32}\| f\|_{L^6}^6 \le \frac{1}{2}\| \del_x f\|_{L^2}^2\cdot
\l( \frac{ \sqrt{\gamma} }{2\pi}\| f\|_{L^2}^2\r)^2,
\end{align*}
one can prove that $-\infty <-\nu (c,m)<0$ for $m\in (0,\tfrac{2\pi}{\sqrt{\gamma}})$. Other parts in the proof work without any changes. We note that the condition $m\in (0,\tfrac{2\pi}{\sqrt{\gamma}})$ is essential to prove $-\infty <-\nu (c,m)$, so that we need to restrict our approach to the case of negative velocity. 
%

\section*{Acknowledgments}
The results of this paper were mostly obtained when the author was a PhD student at Waseda University. The author would like to thank his thesis adviser Tohru Ozawa for constant encouragements.
The author is also grateful to Masahito Ohta for fruitful discussions, and to Noriyoshi Fukaya for helpful comments on the first manuscript. This work was supported by JSPS KAKENHI Grant Numbers JP17J05828, JP19J01504, and Top Global University Project, Waseda University.


\end{document}